\documentclass[11pt,USletter]{article}
\usepackage{fullpage}

\usepackage[utf8]{inputenc}
\usepackage{amsthm,amssymb}
\usepackage[leqno]{amsmath}
\usepackage{thmtools}
\usepackage{subcaption,thm-restate}
\usepackage[mathcal,mathscr]{eucal}
\usepackage{epsfig,graphicx,graphics,color}
\usepackage{caption}
\usepackage[position=b]{subcaption}
\usepackage{enumerate}
\usepackage[sort,nocompress]{cite}
\usepackage{hyperref}
\hypersetup{
    colorlinks=true,       
    linkcolor=blue,        
    citecolor=red,         
    filecolor=magenta,     
    urlcolor=cyan,         
    linktocpage=true
}

\theoremstyle{plain}
\newtheorem{thm}{Theorem}
\newtheorem{lem}[thm]{Lemma}

\newtheorem{conj}[thm]{Conjecture}

\makeatletter
\newcommand{\leqnomode}{\tagsleft@true}
\newcommand{\reqnomode}{\tagsleft@false}
\makeatother

\newcommand{\cB}{\mathcal{B}}
\newcommand{\cI}{\mathcal{I}}
\newcommand{\cJ}{\mathcal{J}}
\newcommand{\cC}{\mathcal{C}}

\newcommand{\cups}{\cup \dots \cup} 

\newcommand{\overbar}[1]{\mkern 3mu\overline{\mkern-3mu#1\mkern-3mu}\mkern 3mu}
\makeatletter
\newcommand{\dbloverline}[1]{\overbar{\dbl@overline{#1}}}
\newcommand{\dbl@overline}[1]{\mathpalette\dbl@@overline{#1}}
\newcommand{\dbl@@overline}[2]{%
  \begingroup
  \sbox\z@{$\m@th#1\overbar{#2}$}%
  \ht\z@=\dimexpr\ht\z@-1.5\dbl@adjust{#1}\relax
  \box\z@
  \ifx#1\scriptstyle\kern-\scriptspace\else
  \ifx#1\scriptscriptstyle\kern-\scriptspace\fi\fi
  \endgroup
}
\newcommand{\dbl@adjust}[1]{%
  \fontdimen8
  \ifx#1\displaystyle\textfont\else
  \ifx#1\textstyle\textfont\else
  \ifx#1\scriptstyle\scriptfont\else
  \scriptscriptfont\fi\fi\fi 3
}
\makeatother

\def\final{0}  
\def\iflong{\iffalse}
\ifnum\final=0  
\usepackage[dvipsnames]{xcolor}
\newcommand{\knote}[1]{{\color{red}[{\tiny \textbf{Kristóf:} \bf #1}]\marginpar{\color{red}*}}}
\newcommand{\tnote}[1]{{\color{blue}[{\tiny \textbf{Tamás:} \bf #1}]\marginpar{\color{blue}*}}}
\else 
\newcommand{\knote}[1]{}
\newcommand{\tnote}[1]{}
\fi

\title{Rainbow and monochromatic circuits and cuts in binary matroids}

\author{
Kristóf Bérczi\thanks{MTA-ELTE Momentum Matroid Optimization Research Group and MTA-ELTE Egerv\'ary Research Group, Department of Operations Research, E\"otv\"os Lor\'and University, Budapest, Hungary. Email: \texttt{kristof.berczi@ttk.elte.hu}.}
\and
Tamás Schwarcz\thanks{MTA-ELTE Momentum Matroid Optimization Research Group, Department of Operations Research, E\"otv\"os Lor\'and University, Budapest, Hungary. Email: \texttt{tamas.schwarcz@ttk.elte.hu}.}
}

\begin{document}
\maketitle

\begin{abstract}
Given a matroid together with a coloring of its ground set, a subset of its elements is called rainbow colored if no two of its elements have the same color. We show that if an $n$-element rank $r$ binary matroid $M$ is colored with exactly $r$ colors, then $M$ either contains a rainbow colored circuit or a monochromatic cut. As the class of binary matroids is closed under taking duals, this immediately implies that if $M$ is colored with exactly $n-r$ colors, then $M$ either contains a rainbow colored cut or a monochromatic circuit. As a byproduct, we give a characterization of binary matroids in terms of reductions to partition matroids. 

Motivated by a conjecture of B\'erczi et al., we also analyze the relation between the covering number of a binary matroid and the maximum number of colors or the maximum size of a color class in any of its rainbow circuit-free colorings. For simple graphic matroids, we show that there exists a rainbow circuit-free coloring that uses each color at most twice only if the graph is $(2,3)$-sparse, that is, it is independent in the $2$-dimensional rigidity matroid. Furthermore, we give a complete characterization of minimally rigid graphs admitting such a coloring.

\medskip

\end{abstract}

\section{Introduction} \label{sec:intro}

Matroids play a crucial role in optimization problems due to their high level of abstraction that enables them to represent various combinatorial objects. In many cases, however, the underlying matroidal structure is still difficult to work with, and a simpler matroid is required that provides a `good' approximation of the original one. In this paper, we consider colorings of matroids that defines simpler partition matroids providing such an approximation. 

Given a matroid together with a coloring of its ground set, a subset of its elements is called \textbf{rainbow colored} if it does not contain two elements of the same color. Accordingly, a coloring is called \textbf{rainbow circuit-free} or \textbf{rainbow cut-free} if no circuit or cut is rainbow colored, respectively.

Every loopless matroid of rank $r$ has a rainbow circuit-free coloring with exactly $r$ colors by the following construction: for $i=r,r-1\dots,1$, let $S_i$ be a cut of the matroid $M|(S-\bigcup_{j>i} S_j)$, where $|$ denotes the restriction operator. Note that this way $S_i$ is a cut of the rank $i$ matroid $M|(S_1 \cups S_i)$, motivating the reversed ordering of the indices. We call colorings obtained by this construction \textbf{standard}.\footnote{Standard colorings were previously defined for graphs in \cite{hoffman2019rainbow}; we extend this notion for arbitrary matroids.} We show in Section~\ref{sec:standard} that standard colorings are indeed rainbow circuit-free. Obviously, if the matroid contains a loop then no rainbow circuit-free coloring exists. Therefore every matroid considered in the paper is assumed to be loopless without explicitly mentioning this. Nevertheless, parallel elements might exist. 

It is worth mentioning that not every rainbow circuit-free coloring is standard. To illustrate this, consider the uniform matroid $U_{2,4}$ of rank two on four elements. As every circuit consists of three elements, any coloring of the matroid with two colors is rainbow circuit-free. However, if both color classes contain two elements then no monochromatic cut exists, so the coloring cannot be obtained by the above algorithm.

\paragraph{\it Previous work}

Simplifications of matroids appeared already in the late 60's when Crapo and Rota \cite{crapo1968combinatorial} introduced the notion of weak maps. Following the terminology of \cite{berczi2019list}, given two matroids $N$ and $M$ on the same ground set, $N$ is said to be a \textbf{reduction} of $M$ if every independent set of $N$ is also independent in $M$. If, furthermore, the rank of $N$ coincides with that of $M$, then $N$ is a \textbf{rank-preserving reduction} of $M$. In terms of weak maps, a reduction corresponds to the identity weak map on the common ground set of the two matroids.

Lucas \cite{lucas1974properties,lucas1975weak} studied rank-preserving reductions of binary matroids together with the behavior of certain invariants under reductions, such as the Tutte polynomial, Whitney numbers, or the M\"obius function. For further results and remarks on weak images, we refer the interested reader to \cite{kung1986weak}.

Hoffman et al. \cite{hoffman2019rainbow} considered rainbow cycle-free edge colorings of finite graphs. They proved that every rainbow cycle-free coloring of a connected graph on $n$ vertices with $n-1$ colors implies a monochromatic edge cut in $G$. They also showed that every such coloring can be obtained by taking and removing cuts from the graph sequentially -- in an analogous way as described in the introduction.

In a recent paper, Im et al. \cite{im2020matroid} studied the so-called matroid intersection cover problem, a special case of set cover where the sets are derived from the intersection of matroids. Their approach is based on partition decompositions of matroids, a generalization of reductions to partition matroids. They gave polynomial-time algorithms to compute such partition decompositions for several matroid classes that commonly arise in combinatorial optimization problems. 

In an independent work \cite{berczi2019list}, B\'erczi et al. investigated the list coloring number of the intersection of two matroids. A key tool in their approach was finding a reduction of a matroid to a partition matroid without increasing its coloring number too much. They proved that such a reduction exists for paving, graphic matroids, and gammoids -- for all those cases, they verified the existence of a reduction into a partition matroid with coloring number at most twice that of the original matroid.

\paragraph{\it Our results}

Rank-preserving reductions define a natural partial order on the set of matroids. It is not difficult to check that the minimal elements with respect to this partial order correspond to partition matroids with upper bounds one on the partition classes, therefore the structurally simplest approximations of any matroid fall in this class. This motivates the investigation of reductions from an arbitrary matroid to a partition matroid; we will see that reductions of such form correspond to rainbow circuit-free colorings. In this paper, we concentrate on binary matroids.

Our first main result is the following alternative theorem.  

\begin{restatable}{thm}{thmrainbow}
\label{thm:rainbow}
Let $M$ be an $n$-element binary matroid of rank $r$.  
\begin{enumerate}
    \item If the ground set of $M$ is colored with exactly $r$ colors, then $M$ either contains a rainbow circuit or a monochromatic cut.
    \item If the ground set of $M$ is colored with exactly $n-r$ colors, then $M$ either contains a rainbow cut or a monochromatic circuit.
\end{enumerate}
\end{restatable}

As a consequence, we show that the structure of rainbow circuit-free and rainbow cut-free colorings is rather restricted.

\begin{restatable}{ocor}{corrainbow}
\label{cor:rainbow}
Let $M$ be an $n$-element binary matroid of rank $r$. 
\begin{enumerate}
    \item If the ground set of $M$ is colored with exactly $r$ colors and $M$ has no rainbow circuit, then one of the color classes is a cut of $M$. \label{eq:cor1}
    \item If the ground set of $M$ is colored with exactly $n-r$ colors and $M$ has no rainbow cut, then one of the color classes is a circuit of $M$.
\end{enumerate}
\end{restatable}

\begin{restatable}{ocor}{corstandard}
\label{cor:standard} 
Every rainbow circuit-free coloring of a rank $r$ binary matroid with $r$ colors is standard. In particular, one of the color classes consists of parallel elements.
\end{restatable}

Somewhat surprisingly, the reverse implication also holds, therefore we get a characterization of binary matroids.

\begin{restatable}{thm}{thmstandard}
\label{thm:standard}
A matroid of rank $r$ is binary if and only if each of its rainbow circuit-free colorings with exactly $r$ colors is standard.
\end{restatable}

We provide several applications of the above results. As a local counterpart of strongly base orderability, we introduce the notion of  locally strongly base orderable basis-pairs and characterize those for binary matroids.

\begin{restatable}{thm}{thmlsbo} 
\label{thm:lsbo}
Let $M$ be a binary matroid and let $B_1$ and $B_2$ be bases of $M$ such that $S=B_1\cup B_2$. Then $B_1$ and $B_2$ are locally strongly base orderable if and only if $M/(B_1\cap B_2)$ has a standard coloring with color classes of size two. In particular, if $B_1$ and $B_2$ are disjoint and $M$ is simple, then no such bijection exists.
\end{restatable}

Note that the assumption $S=B_1\cup B_2$ is not restrictive as the restriction of $M$ to $B_1\cup B_2$ is also binary.

An interesting problem of matroidal optimization is to cover the ground set of two matroids by a minimum number of common independent sets. Recently, several works attempted to attack this problem through rainbow circuit-free colorings with small color classes. We focus on the relation between the covering number of a matroid and the maximum size of a color class, where the covering number of the matroid is defined as the minimum number of independent sets needed to cover its ground set. 

\begin{restatable}{thm}{thmnorpbin} 
\label{thm:norpbin}
For every positive integer $g$, there exists a $2$-coverable binary matroid $M$ such that in any rank-preserving rainbow circuit-free coloring of $M$, one of the colors is used at least $g$ times.
\end{restatable}

We also prove a general statement that provides lower and upper bounds on the number of colors used in a rainbow circuit-free coloring with bounded sized color classes.

\begin{restatable}{thm}{thmrank} 
\label{thm:rank}
Let $M$ be a matroid, $g\geq 3$ be a positive integer, and $\cC$ be a family of circuits of $M$ such that for every $0\leq i \leq g-1$ and for every $X\subseteq S,|X|=i$ we have $|\{C\in\cC:\ |X\cap C|\geq 2\}|\leq i-1$. Then in any rainbow circuit-free coloring of $M$ such that each color is used at most $(g-1)$-times, the number of different colors is between $|\cC|/(g-2)$ and $|S|-|\cC|$.    
\end{restatable}

As a consequence, we show that if every color can occurs at most three times then the number of colors used in a rainbow circuit-free coloring might be only a fraction of the rank of the original matroid.

\begin{restatable}{ocor}{corrank}
Given any positive integer $r$, there exists a $2$-coverable connected binary matroid of rank $r$ such that in any rainbow circuit-free coloring of $M$ such that each color is used at most three times, the number of colors is at most $\lceil 6/7\cdot r\rceil$.
\end{restatable}

Finally, we consider rainbow circuit-free colorings that do not increase the covering number of the matroid. For simple graphic matroids, we show that there exists a rainbow circuit-free coloring  that uses each color at most twice only if the graph is independent in the $2$-dimensional rigidity matroid.

\begin{restatable}{thm}{thmsparse}\label{thm:sparse}
If a simple graphic matroid has a rainbow circuit-free coloring such that each color is used at most twice, then the corresponding graph is $(2,3)$-sparse.
\end{restatable}

Furthermore, we give a complete characterization of minimally rigid graphs that admit such a coloring in terms of Henneberg operations.

\begin{restatable}{thm}{thmho}
Let $G=(V,E)$ be a simple graph such that $|E|=2\cdot |V|-3$. Then the graphic matroid of $G$ has a rainbow circuit-free coloring using each color at most twice if and only if $G$ can be constructed from $K_2$ by a sequence of (H0) operations.
\end{restatable}

\section{Preliminaries} \label{sec:preliminaries}

\paragraph{\it General notation}

Let $G=(V,E)$ be an undirected graph, $X\subseteq V$ be a subset of vertices, and $F\subseteq E$ be a subset of edges. We denote the \textbf{set of edges spanned by $X$} by $E[X]$, while the \textbf{set of edges incident to a vertex} $v\in V$ is denoted by $\delta(v)$. 
The graphs obtained by deleting $X$ or $F$ are denoted by $G-X$ and $G-F$, respectively. An inclusionwise minimal subset of edges whose deletion increases the number of connected components of $G$ is called an \textbf{elementary cut}.

Given a ground set $S$ together with subsets $X,Y\subseteq S$, 
the \textbf{difference} of $X$ and $Y$ is denoted by $X-Y$. If $Y$ consists of a single element $y$, then $X-\{y\}$ and $X\cup\{y\}$ are abbreviated as $X-y$ and $X+y$, respectively.

\paragraph{\it Matroids}

Matroids were introduced as abstract generalizations of linear independence by Whitney \cite{whitney1992abstract} and independently by Nakasawa \cite{nishimura2009lost}. A \textbf{matroid} $M=(S,\cI)$ is defined by its \textbf{ground set} $S$ and its \textbf{family of independent sets} $\cI\subseteq 2^S$ that satisfies the \textbf{independence axioms}: (I1) $\emptyset\in\cI$, (I2) $X\subseteq Y,\ Y\in\cI\Rightarrow X\in\cI$, and (I3) $X,Y\in\cI,\ |X|<|Y|\Rightarrow\exists e\in Y-X\ s.t.\ X+e\in\cI$. The maximum size of an independent subset of a set $X$ is the \textbf{rank} of $X$ and is denoted by $r_M(X)$. The maximal independent subsets of $S$ are called \textbf{bases}. A \textbf{circuit} is an inclusionwise minimal non-independent set, while a \textbf{loop} is a circuit consisting of a single element. The matroid is \textbf{connected} if for any two elements $e,f\in S$ there exists a circuit containing both. A \textbf{cut} is an inclusionwise minimal set $X\subseteq S$ that intersects every basis. Although in general a cut and a circuit can intersect each other in an odd number of elements, the size of their intersection cannot be one. Two non-loop elements $e,f\in S$ are \textbf{parallel} if $r_M(\{e,f\})=1$. A set $X\subseteq S$ is \textbf{closed} or is a \textbf{flat} if $r_M(X+e)>r_M(X)$ for every $e\in S-X$. A flat of rank $r_M(S)-1$ is called a \textbf{hyperplane}. 

The \textbf{restriction} of a matroid $M=(S,\cI)$ to a subset $S'\subseteq S$ is again a matroid $M|S'=(S',\cI')$ with independence family $\cI'=\{I\in\cI:\ I\subseteq S'\}$. The \textbf{contraction} of a flat $S''\subseteq S$ results in a matroid $M/S''=(S-S'',\cI'')$ where $\cI''=\{I\in\cI:\ I\subseteq S-S'',|I|=r_M(S''\cup I)-r_M(S'')\}$. A matroid $N$ that can be obtained from $M$ by a sequence of restrictions and contractions is called a \textbf{minor} of $M$; in this case $M$ \textbf{contains $N$ as a minor}. The \textbf{dual} of $M$ is the matroid $M^*=(S,\cI^*)$ where $\cI^*=\{X\subseteq S:\ S-X\ \text{contains a basis of $M$}\}$. The \textbf{direct sum} $M_1\oplus M_2$ of matroids $M_1=(S_1,\cI_1)$ and $M_2=(S_2,\cI_2)$ on disjoint ground sets is a matroid $M=(S_1\cup S_2,\mathcal{I})$ whose independent sets are the disjoint unions of an independent set of $M_1$ and an independent set of $M_2$. 

The minimum size of a circuit is called the \textbf{girth}, while the minimum size of a cut is called the \textbf{co-girth} of the matroid. It is not difficult to check that a circuit of a matroid $M$ is a cut in the dual matroid $M^*$ and vice versa, thus the girth and co-girth of $M$ are the same as the co-girth and girth of $M^*$, respectively. 

For non-negative integers $r\leq n$, the \textbf{uniform matroid} $U_{r,n}$ is defined on an $n$-element set by setting every set of size at most $r$ to be independent. A matroid is \textbf{binary} if there exists a family of vectors from a vector space over the finite field $GF(2)$ whose linear independence relation is the same as that of the matroid. Binary matroids form a minor-closed class that is also closed for taking dual. Tutte \cite{tutte1958homotopyi,tutte1958homotopyii} gave the following forbidden minor characterization of binary matroids. 

\begin{thm}[Tutte] \label{thm:tutte}
A matroid is binary if and only if it has no $U_{2,4}$-minor.
\end{thm}

A \textbf{partition matroid} is a matroid $N=(S,\cJ)$ such that $\cJ=\{X\subseteq S:\ |X\cap S_i|\leq 1\ \text{for $i=1,\dots,q$}\}$ for some partition $S=S_1\cup\dots\cup S_q$.\footnote{In general, the upper bounds might be different for the different partition classes. As all the partition matroids used in the paper have all-ones upper bounds, we make this restriction without explicitly mentioning it.} For a graph $G=(V,E)$, the \textbf{graphic matroid} $M=(E,\cI)$ of $G$ is defined as $\cI=\{F\subseteq E:\ F\ \text{does not contain a cycle}\}$, that is, a subset $F\subseteq E$ is independent if it is a forest. The dual of the graphic matroid is called the \textbf{co-graphic} matroid of $G$. Note that a subset $F\subseteq E$ is a circuit of the co-graphic matroid if and only if $F$ is an elementary cut of $G$.

Given two matroids $M=(S,\cI)$ and $N=(S,\cJ)$ over the same ground set, we say that $N$ is a \textbf{reduction} of $M$ if every independent set of $N$ is also independent in $M$, that is, $\cJ\subseteq \cI$. We denote such a relation by $N\preceq M$. If, furthermore, the rank of $N$ coincides with that of $M$, then $N$ is a \textbf{rank-preserving reduction} of $M$ which is denoted by $N\preceq_r M$.

Lucas \cite{lucas1974properties} studied rank-preserving reductions of binary matroids, and proved the following result that we will rely on.

\begin{thm}[Lucas] \label{thm:Lucas} 
Let $M$ and $N$ be binary matroids such that $N\preceq_r M$ and $N \ne M$. Then there is a non-empty flat $F$ of $M$ such that \[N \preceq_r M|F \oplus M/F \preceq_r M. \]
\end{thm}

\paragraph{\it Colorings and reducibility}

There is a one-to-one correspondence between rainbow circuit-free colorings of a matroid and its reductions to partition matroids. Indeed, if we pick one element from each color class of any rainbow circuit-free coloring, then the resulting set is independent -- in other words, the partition matroid defined by the color classes provides a reduction of the matroid. Vice versa, taking the partition classes of any reduction to a partition matroid as color classes gives a rainbow circuit-free coloring. By the above, the terms `rainbow circuit-free colorings' and `reductions to partition matroids' are interchangeable and are used alternatingly throughout the paper depending on the context. In particular, a rank-preserving rainbow circuit-free coloring corresponds to a rank-preserving reduction to a partition matroid.

\section{Rainbow circuit-free colorings}

\subsection{Standard colorings} \label{sec:standard}

Recall that a coloring was defined to be standard if, after possibly reindexing the color classes, $S_i$ is a cut of the matroid $M|(S_1 \cups S_i)$ for $i=1,\dots,r$. We show that every standard coloring is rainbow circuit-free, together with some equivalent formulations of standard colorings which will be used later.

\begin{lem} \label{lem:equiv} Let $M$ be a matroid of rank $r$ and $S =S_1 \cups S_r$ be  partition of its ground set into $r$ color classes. The followings are equivalent:
\begin{enumerate}[(i)]
    \item $S_i$ is a cut of $M|(S_1 \cups S_i)$ for $i = 1, \dots, r$, \label{eq:i}
    \item the coloring is rainbow circuit-free and $r_M(S_1 \cups S_i) = i$ for $i=1, \dots r$, \label{eq:ii}
    \item each $S_i$ is non-empty and $S_1 \cups S_i$ is closed in $M$ for $i=1,\dots, r$. \label{eq:iii}
\end{enumerate}
\end{lem}
\begin{proof}
First we show that \eqref{eq:i} implies \eqref{eq:ii}. As the deletion of a cut decreases the rank of a matroid by exactly one, $r_M(S_1 \cups S_i) = i$ indeed holds for $i=1 \dots, r$. For any circuit $C$ of $M$, let $i$ be the largest index for which $C\cap S_i$ is non-empty. Then $C \subseteq S_1 \cups S_i$ and $S_i$ is a cut of $M|(S_1 \cups S_i)$, hence $|C \cap S_i| \ge 2$. Therefore no circuit $C$ is rainbow colored.

Assume next that \eqref{eq:ii} holds. For every $i=1, \dots, r$ and $s \in S-\bigcup_{j=1}^i S_j$, exactly $i+1$ colors appear in $S_1 \cups S_i \cup \{s\}$, hence $r_M(S_1 \cups S_i\cup \{s\}) \ge i+1$ as each rainbow set is independent. Then $r_M(S_1 \cups S_i \cup \{s\}) > r_M(S_1 \cups S_i)$, hence $S_1 \cups S_i$ is closed, proving that \eqref{eq:ii} implies \eqref{eq:iii}.

Finally, suppose that \eqref{eq:iii} holds. As $S_1 \subset (S_1 \cup S_2) \subset \dots \subset (S_1 \cups S_r)$ is a strictly increasing chain of flats, we have $1 \le r_M(S_1) < r_M(S_1 \cup S_2) < \dots < r_M(S_1 \cups S_r) = r$, hence $r_M(S_1 \cups S_i) = i$ for $i=1,\dots, r$. As $S_1 \cups S_{i-1}$ is also a closed set of $M|(S_1 \cups S_i)$, $S_i$ is the complement of a hyperplane of $M|(S_1 \cups S_i)$, that is, $S_i$ is a cut of this matroid, proving that \eqref{eq:iii} implies \eqref{eq:i}.
\end{proof}

Note that \eqref{eq:iii} provides the following algorithm for constructing standard colorings: For $i=1, \dots, r$, pick an element $s_i \in S-(S_1 \cups S_{i-1})$, and let $S_i$ be the closure of $S_1 \cups S_{i-1} \cup \{s_i\}$ in $M$. As flats of rank 1 are exactly the classes of parallel elements, the following is yet another way to formulate this algorithm: Let $S_i$ be a class of parallel elements in $M/(S_1 \cups S_{i-1})$ for $i=1,\dots, r$.

\subsection{Rainbow and monochromatic circuits and cuts} \label{sec:rainbow}

Based on Lucas' theorem, we turn to the proof of our first result.

\thmrainbow*
\begin{proof}
As the class of binary matroids is closed under taking duals and the cuts of a matroid coincide with the circuits of the dual matroid, it suffices to prove the first statement. We proceed by induction on the size of the ground set $S$ of $M$. Let $N$ be the partition matroid defined by the coloring of $M$. Recall that $N \preceq_r M$ holds. If $N=M$, then each color class is a monochromatic cut of $M$. Otherwise, by Theorem~\ref{thm:Lucas} there exists a non-empty flat $F$ of $M$ such that $N \preceq_r M|F \oplus M/F \preceq_r M$. From $N \preceq M|F \oplus M/F$ it follows that $N|F \preceq M|F$ and $N|(S-F) \preceq M/F$. This implies that 
\begin{align*}
r 
{}&{}= 
r_N(S) \\
{}&{}\le
r_N(F)+r_N(S-F)\\
{}&{}\le
r_M(F)+r_{M/F}(S-F)\\
{}&{}= 
r_M(S) = r,
\end{align*}
hence $r_N(F) = r_M(F)$ and $r_N(S-F) = r_{M/F}(S-F)$. Therefore we get $N|F \preceq_r M|F$ and $N|(S-F) \preceq_r M/F$. Applying the induction hypothesis to the reduction $N|(S-F) \preceq_r M/F$, we get a monochromatic cut of $M/F$.  As each cut of $M/F$ is also a cut of $M$, the original matroid $M$ has a monochromatic cut as well.
\end{proof}

\corrainbow*
\begin{proof}
By duality it suffices to prove the first statement. If $M$ has no rainbow circuit, then by Theorem~\ref{thm:rainbow} it has a cut $C$ and a color class $C'$ such that $C \subseteq C'$. If $C \subsetneq C'$, then each color class intersects $S-C$. As rainbow colored sets of size $r$ are bases of $M$, this implies that there is a basis of $M$ which is disjoint from the cut $C$, a contradiction. Hence $C=C'$, that is, the color class $C'$ is a cut of $M$.
\end{proof}

\corstandard* 
\begin{proof}
As the restriction of a binary matroid is binary again, Corollary~\ref{cor:rainbow} implies that the color classes $S_1, \dots, S_r$ can be indexed in such a way that $S_i$ is a cut of $M|(S-\bigcup_{j>i} S_j)$. In particular, the color class $S_1$ consists of parallel elements of $M$.
\end{proof}

\subsection{Characterization of binary matroids} \label{sec:char}

The aim of this section is to prove Theorem~\ref{thm:standard}. The proof is based on the excluded-minor characterization of binary matroids, so we first prove two lemmas about standard colorings of minors of matroids.

\begin{lem} \label{lem:restriction}
Consider a standard coloring of a matroid $M$ and let $Z \subseteq S$ be a subset such that exactly $r_M(Z)$ colors appear in $Z$. Then the coloring restricted to $Z$ is a standard coloring of $M|Z$.
\end{lem}
\begin{proof}
Let $S_1, \dots, S_r$ denote the color classes such that $S_1 \cup \dots \cup S_i$ is closed in $M$ for $i=1, \dots, r$ (such an ordering exists by Lemma~\ref{lem:equiv}). Let $1 \le i_1 < i_2 < \dots < i_{r_M(Z)} \le r$ denote the indices of colors appearing in $Z$. Then $(S_{i_1} \cap Z) \cup (S_{i_2} \cap Z) \cups (S_{i_j} \cap Z) =(S_1 \cup S_2 \cups S_{i_j}) \cap Z$ is closed in $M|Z$ for $j=1, \dots, r_M(Z)$, hence the coloring of $M|Z$ with color classes $Z\cap S_{i_1}, \dots, Z\cap S_{i_{r_M(Z)}}$ is standard. 
\end{proof}

The proof of Theorem~\ref{thm:standard} will rely on the following lemma. 

\begin{lem} \label{lem:minor}
Let $F$ be a flat of a matroid $M$. If either the restriction $M|F$ or the contraction $M/F$ has a non-standard rank-preserving coloring without rainbow circuits, then so does $M$.
\end{lem}
\begin{proof}
Let $N=M/F \oplus M|F$ and $r = r_M(S)=r_N(S)$. Consider a coloring of either $M|F$ or $M/F$ provided by the assumption of the lemma and extend it to $S$ by a standard coloring of $M/F$ or $M|F$, respectively, such that the color sets used on $M|F$ and $M/F$ are disjoint. This way we get a rainbow circuit-free coloring of $N$ with exactly $r$ colors. As $N \preceq_r M$, this also a rainbow circuit-free coloring of $M$. 

The restriction of the coloring to either $N|F$ or $N|(S-F)$ is non-standard by the construction, hence, by Lemma~\ref{lem:restriction}, it is a non-standard coloring of $N$. Suppose for contradiction that it is a standard coloring of $M$. By Lemma~\ref{lem:equiv}, there is an ordering $S_1, \dots, S_r$ of the color classes such that $r_M(S_1 \cups S_i) = i$ for $i=1, \dots, r$. As $N \preceq M$, this implies that $r_N(S_1 \cups S_i) \le r_M(S_1 \cups S_i) = i$.  As it is a rainbow circuit-free coloring of $N$, we have $r_N(S_1 \cups S_i) \ge i$ as well, thus $r_N(S_1 \cups S_i) = i$ for $i=1, \dots, r$. Therefore, it is a standard coloring of $N$ by Lemma~\ref{lem:equiv}, a contradiction.  
\end{proof}

We are ready to prove the theorem. 

\thmstandard*
\begin{proof}

By Corollary~\ref{cor:standard}, it suffices to show that a non-binary matroid $M$ has a non-standard rainbow circuit-free coloring with exactly $r$ colors. As $M$ is non-binary, by Theorem~\ref{thm:tutte}, it has a minor isomorphic to $U_{2,4}$. That is, there exists a flat $F$ of $M$ and a subset $T \subseteq S-F$ such that $(M/F)|T$ is isomorphic to $U_{2,4}$. 
Let $G$ denote the closure of $T$ in $M/F$. As $T$ has rank 2 in $M/F$, $N = (M/F)|G$ is a matroid of rank 2. Let $G = P_1 \cups P_q$ denote the partition of $G$ into classes of parallel elements of $N$. As $N|T$ is isomorphic to $U_{2,4}$, we have $q\ge 4$. Color $G$ with two colors such that the color classes are $P_1 \cup P_2$ and $P_3 \cups P_q$. Since $N$ has rank 2 and parallel elements receive the same color, this is a rank-preserving coloring of $N$ without rainbow circuits. Moreover, neither $P_1 \cup P_2$ nor $P_3 \cups P_q$ is closed in $N$, hence it is a non-standard coloring of $N$. Applying Lemma~\ref{lem:minor} to $M/F$ and to the restriction $N=(M/F)|G$, we get that $M/F$ has a non-standard rank-preserving coloring without rainbow circuits (note that $G$ is closed in $M/F$). Now applying Lemma~\ref{lem:minor} to $M$ and to the contraction $M/F$, we get that $M$ also has a non-standard rank-preserving rainbow circuit-free coloring, concluding the proof of the theorem.  
\end{proof}

\section{Consequences} \label{sec:consequences}

\subsection{Binary matroids and strongly base orderability}

Given two bases $B_1$ and $B_2$ of a matroid, there exists a bijection $\varphi:B_1\to B_2$ with the property that $B_1-x+\varphi(x)$ is a basis for each $x\in B_1-B_2$ (see e.g. \cite[Theorem 5.3.4]{frank2011connections}). A strengthening of this observation was given by Greene and Magnanti \cite{greene1975some}.

\begin{thm}[Greene and Magnanti] \label{thm:gm}
Let $B_1$ and $B_2$ be two bases of a matroid $M$ and let $\{X_1,\dots,X_q\}$ be a partition of $B_1$. Then there is a partition $\{Y_1,\dots,Y_q\}$ of $B_2$ for which $(B_1-X_i)\cup Y_i$ is a basis for $i=1,\dots,q$.
\end{thm}

Let us make two important observations. First, the roles of $B_1$ and $B_2$ in Theorem~\ref{thm:gm} are not symmetric, hence $(B_2-Y_i)\cup X_i$ might not be a basis of $M$ for some $i$. Second, only single exchanges are possible, hence $(B_1-(X_i\cup X_j))\cup (Y_i\cup Y_j)$ might not be a basis for some pair $i\neq j$.

Strongly base orderable matroids form a class of matroids with distinguished structural properties overcoming the above difficulties. A matroid $M$ with basis-family $\cB$ is \textbf{strongly base orderable} if for any two bases $B_1,B_2\in\cB$ there exists a bijection $\varphi:B_1\to B_2$ with the property that 
\leqnomode
\begin{equation}
(B_1-X)\cup \varphi(X)\in\cB\quad \text{for}\ X\subseteq B_1. \tag{SBO} \label{eq:sbo}
\end{equation}
\reqnomode
Several matroid classes that appear in combinatorial optimization problems are strongly base orderable, such as gammoids (and so partition, laminar, or transversal matroids), while others do not possess such nice characteristics, e.g. graphic or paving matroids.

As a local counterpart of strongly base orderability, we say that two bases $B_1$ and $B_2$ of a matroid are \textbf{locally strongly base orderable} if there exists a bijection $\varphi:B_1\to B_2$ satisfying \eqref{eq:sbo}. With the help of Theorem~\ref{thm:standard}, we give a characterization of locally strongly base orderable basis-pairs of binary matroids.

\thmlsbo*
\begin{proof}
Let $M':=M/(B_1\cap B_2)$ and let us denote the rank of $M'$ by $r'$. Assume first that $\varphi:B_1\to B_2$ satisfies \eqref{eq:sbo}. Consider the coloring of $M'$ with $r'$ color classes of the form $\{x,\varphi(x)\}$ for $x\in B_1-B_2$. By \eqref{eq:sbo}, if we pick exactly one element from every color class then the resulting set is a basis of $M'$, therefore the coloring is rainbow circuit-free. As $M'$ is binary, the coloring is standard by Theorem~\ref{thm:standard}.

For the opposite direction, consider a standard coloring of $M'$ with color classes of size two. By Lemma~\ref{lem:equiv}, there is an ordering $S_1,\dots,S_{r'}$ of the color classes such that $S_j$ consists of parallel elements of $M'/(\bigcup_{k=1}^{j-1} S_k)$ for $j=1,\dots,r'$. As both $B_1-B_2$ and $B_2-B_1$ are bases of $M'$, we get that $|B_i\cap S_j|=1$ for $i=1,2$ and $j=1,\dots,r'$. Let $S_j=\{x^1_j,x^2_j\}$ where $x^i_j\in B_i$, and define
\begin{equation*}
    \varphi(x)=\begin{cases}
    x & \text{if $x\in B_1\cap B_2$},\\
    x^2_j & \textit{if $x=x^1_j$ for some $j\in\{1,\dots,r'\}$}.
    \end{cases}
\end{equation*}
We claim that $\varphi$ satisfies \eqref{eq:sbo}. Indeed, for an arbitrary subset $X\subseteq B_1$, the set $(B_1-X)\cup \varphi(X)$ contains exactly one element from each of the $S_j$s plus the intersection $B_1\cap B_2$. As the coloring was standard for $M'$, these elements form a basis of $M$.

Now we turn to the second half of the theorem. By the first half, there is a standard coloring of $M$ in which each color class has size two. By Corollary~\ref{cor:rainbow}, in every standard coloring of $M$ there exists a color class consisting of parallel elements. However, the matroid is assumed to be simple, a contradiction. 
\end{proof}

\subsection{Covering number and reduction}

The \textbf{covering number} of a matroid $M=(S,\cI)$ is the minimum number of independent sets from $\cI$ needed to cover the ground set $S$.\footnote{The covering number is sometimes called the \emph{coloring} or \emph{chromatic number} of $M$ in the literature. However, it corresponds to a different concept of coloring than the one discussed in the present paper, therefore we use the \emph{covering} version throughout.} We say that a matroid is \textbf{$k$-coverable} if its covering number is at most $k$. The notion of covering number can be straightforwardly extended to matroid intersection: given two matroids $M_1=(S,\cI_1)$ and $M_2=(S,\cI_2)$ over the same ground set, the \textbf{covering number of their intersection} is the minimum number of common independent sets needed to cover $S$.  

B\'erczi et al. \cite{berczi2019list} and independently Im et al. \cite{im2020matroid} investigated the existence of reductions to partition matroids that increase the covering number by only a constant factor. In particular, the following conjecture was proposed in \cite[Conjecture 8]{berczi2019list}.

\begin{conj}[B\'erczi et al.] \label{conj:red}
Every $k$-coverable matroid can be reduced to a $2\cdot k$-coverable partition matroid.
\end{conj}

The motivation of the conjecture is multifold, let us mention here only two illustrious examples. The conjecture would provide a new proof for a celebrated theorem of Aharoni and Berger \cite{aharoni2006intersection} stating that the covering number of the intersection of two matroids $M_1,M_2$ is at most twice the maximum of the covering number of $M_1$ and $M_2$ plus one. Furthermore, by relying on Galvin's list coloring theorem \cite{galvin1995list} for bipartite graphs, the conjecture would imply that even the list covering number of the intersection of two matroids is at most twice its covering number (see \cite{berczi2019list} for further details).

We have already seen that reductions to partition matroids can be identified with rainbow circuit-free colorings. As the covering number of a partition matroid is just the maximum size of the partition classes defining it, this parameter translates into the maximum size of a color class of the coloring. In order to be consistent with the rest of the paper, we follow the coloring terminology, but all results can be rephrased in terms of reductions.

We show that the strengthening of Conjecture~\ref{conj:red} in which the reduction is required to be rank-preserving does not hold. In fact, we prove a much stronger result stating that a constant fraction loss in the rank is unavoidable in certain situations. It is worth mentioning that the results hold already for co-graphic matroids, which is in sharp contrast to the graphic case for which the conjecture was shown to be true even with rank-preserving reduction and with $2\cdot k-1$ in place of $2\cdot k$ \cite{berczi2019list,hoffman2019rainbow,im2020matroid}. In general, we are not aware of any $k$-coverable matroid that does not admit a (not necessarily rank-preserving) $(2\cdot k-1)$-coverable reduction. 

First we show that for certain matroids the covering number of each rank-preserving reduction might be large.

\thmnorpbin*
\begin{proof}
Let $M=(S,\cI)$ be a 2-coverable matroid of co-girth $g$. Such a matroid exists, as shown by the following construction. Every $4$-edge-connected graph $G$ contains two disjoint spanning trees by Nash-Williams' theorem \cite{nash1964decomposition}. The complements of these spanning trees are bases of the co-graphic matroid of $G$ together containing every edge of the graph, thus the co-graphic matroid is $2$-coverable. The co-girth of this matroid is the minimum size of a cycle of $G$. As there exist $4$-edge-connected graphs with arbitrarily large girth (see e.g. \cite{liu1994graphs}), the claim follows.

By Corollary~\ref{cor:rainbow} \eqref{eq:cor1}, any rank-preserving rainbow-free coloring of $M$ has a color class that forms a cut of $M$ and thus has size at least $g$, concluding the proof.
\end{proof}

Theorem~\ref{thm:norpbin} implies that Conjecture~\ref{conj:red} does not hold when restricted to rank-preserving reductions. Based on the following technical statement, we will show that even a constant factor of the rank might be lost if the covering number of the reduction is bounded by $2\cdot k-1$.

\thmrank*
\begin{proof}
Let $S_1,\dots,S_q$ denote the color classes of a rainbow circuit-free coloring of $M$ with $|S_i|\leq g-1$ for $i=1,\dots,q$. Define $\cC_i:=\{C\in\cC:\ |S_i\cap C|\geq 2\}$. As the coloring is rainbow circuit-free, for each $C\in\cC$ there exists a color that appears at least twice in $C$, hence $\cC=\cC_1\cups\cC_q$. By the assumptions of the theorem, $|\cC_i|\leq g-2$, hence
\begin{align*}
    |\cC|
    {}&{}\leq
    |\cC_1|+\dots+|\cC_q|\\
    {}&{}\leq
    (g-2)\cdot q,
\end{align*}
implying the the lower bound $|\cC|/(g-2)$ on $q$. Observe that $|\cC_i|\leq|S_i|-1$, hence
\begin{align*}
    |\cC|
    {}&{}\leq
    |\cC_1|+\dots+|\cC_q|\\
    {}&{}\leq
    |S_1|-1+\dots+|S_q|-1\\
    {}&{}= 
    |S|-q,
\end{align*}
implying the upper bound $|S|-|\cC|$ on $q$.
\end{proof}

\corrank*
\begin{proof}
For $r\leq 6$ the corollary obviously holds as $\lceil 6/7\cdot r\rceil=r$ and every rainbow circuit-free coloring of a matroid uses at most $r$ colors. 

Hence we may assume that $r=7\cdot q+j-1$ for some $q\geq 1$ and $1\leq j\leq 7$. Let $G=(S,T;E)$ denote the complete bipartite graph on seven vertices with $|S|=3$ and $|T|=4$. Let $T=\{u,v,w,z\}$ denote the vertices in the larger vertex class. For $q\geq 1$, let us construct a graph $G_q$ as follows. Take $q$ disjoint copies of $G$, where the vertex classes and edge set of the $i$th copy are denoted by $S_i,T_i=\{u_i,v_i,w_i,z_i\}$ and $E_i$, respectively. Furthermore, connect the subsequent copies by adding the edges $w_iv_{i+1}$ and $z_iu_{i+1}$ to $G_q$ for $i=1,\dots,q-1$). Furthermore, let $G^j_q$ denote the graph obtained from $G_q$ by attaching $j$ vertices to it through vertices $w_q,z_q$ (see Figure~\ref{fig:complete_graph}). 

\begin{figure}
    \centering
    \includegraphics[width=\textwidth]{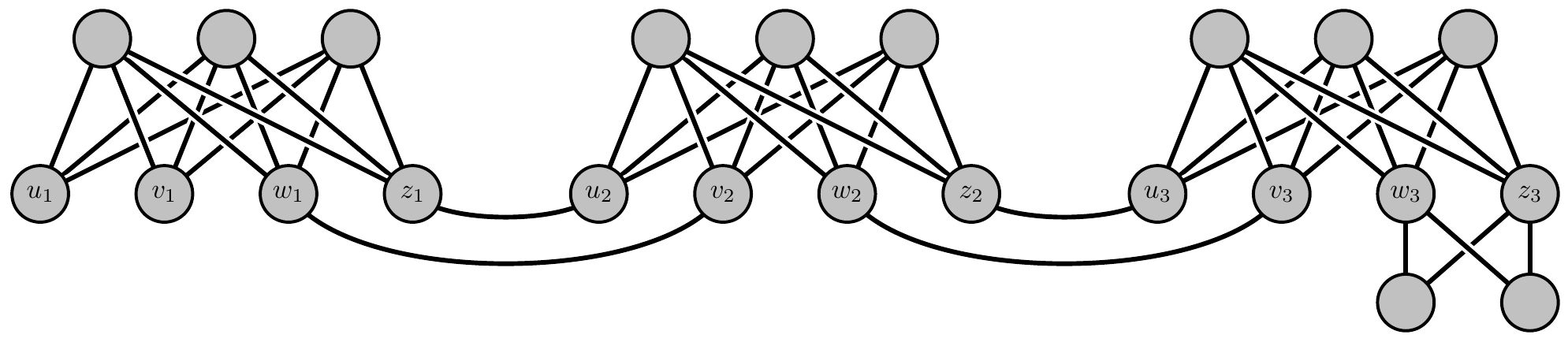}
    \caption{Construction of the graph $G^j_q$ for $q=3$ and $j=2$, corresponding to $r=22$}
    \label{fig:complete_graph}
\end{figure}

Observe that $G^j_q$ is the disjoint union of two spanning trees; this follows from the construction and from the fact that $G$ is the disjoint union of two spanning trees. This implies that the co-graphic matroid $M=(S,\cI)$ of $G^j_q$ is connected, $2$-coverable, and has rank $7\cdot q+j-1=r$. Recall that a circuit of $M$ corresponds to an elementary cut of $G^j_q$. Define
\begin{equation*}
    \cC:=\big\{\delta(v):\ v\ \text{is a vertex of $G^j_q$}\big\}\cup\big\{\{w_iv_{i+1},z_iu_{i+1}\}:\ i=1,\dots,q-1\big\}.
\end{equation*}
We claim that $\cC$ satisfies the conditions of Theorem~\ref{thm:rank} with $g=4$. Indeed, any pair of edges is contained in at most one member of $\cC$. If we consider triples $F$ of edges instead, then $|F\cap C|\geq 2$ can hold either for two members of $\cC$ of the form $\delta(v)$, or a member of $\cC$ of the form $\delta(v)$ and another member of the form $\{w_iv_{i+1},z_iu_{i+1}\}$. This follows from the facts that the shortest cycle in $G$ has length $4$ and the end vertices of the edges participating in the cuts $\{w_iv_{i+1},z_iu_{i+1}\}$ are at a distance of at least two.   

By Theorem~\ref{thm:rank}, the number of colors in any rainbow circuit-free coloring of $M$ is at most 
\begin{align*}
    |S|-|\cC|
    {}&{}=[12\cdot q+2\cdot (q-1)+2\cdot j]-[7\cdot q+j+q-1]\\
    {}&{}=6\cdot q+j-1\\
    {}&{}\leq
    \lceil\frac{6}{7}\cdot (7\cdot q+j-1)\rceil\\
    {}&{}=
    \lceil\frac{6}{7}\cdot r\rceil,
\end{align*}
where the last inequality holds by $1\leq j\leq 7$. This concludes the proof of the corollary.
\end{proof}

\subsection{Graphic matroids}

Conjecture~\ref{conj:red} was verified for the case of graphic matroids even when the reduction is required to be rank-preserving and with $2\cdot k-1$ in place of $2\cdot k$ \cite{berczi2019list,im2020matroid}. However, the characterization of the existence of a (not necessarily rank-preserving) reduction that does not increase the covering number of the matroid remains an intriguing open question.

As an easy case, consider a graph $G$ that can be partitioned into $k$ disjoint spanning trees. Let $M$ be the graphic matroid of $G$ and let $r$ denote its rank. Clearly, the covering number of $M$ is $k$. If $M$ has a rainbow circuit-free coloring that uses each color at most $k$ times, then every color class must have size exactly $k$. By Lemma~\ref{lem:equiv}, there exists an ordering $S_1,\dots,S_r$ of the color classes such that $S_i$ consists of $k$ parallel edges in the graph obtained from $G$ by contracting the edges in $S_1\cup\dots\cup S_{i-1}$. That is, the graph can be reduced to a single vertex by a sequence of contractions of $k$ parallel edges. The reverse direction also holds: if such a sequence of contractions exists for $G$, then coloring the contracted sets by different colors we get a rainbow circuit-free coloring of $M$. 

The problem becomes significantly more difficult if the graph is not the disjoint union of spanning trees. In the followings, we give partial results for the case when the edge set can be decomposed into two forests. A graph $G=(V,E)$ is \textbf{$(2,3)$-sparse} if $|E[X]| \le 2\cdot |X|-3$ holds for $X \subseteq V$, $|X|\ge 2$, and \textbf{$(2,3)$-tight} if, in addition, $|E|=2\cdot |V|-3$ holds. Let $G_0 = (V,E_0)$ denote the complete graph on $V$. The edge sets of $(2,3)$-sparse subgraphs of $G$ form the independent sets of a matroid on $E_0$, the so-called \textbf{rigidity matroid} of $G_0$. The bases of the rigidity matroid are the $(2,3)$-tight graphs on $V$, the so-called \textbf{minimally rigid} or \textbf{Laman graphs}. Minimally rigid graphs are exactly the graphs that can be obtained from $K_2$ (an edge) by the so-called \textbf{Henneberg construction}, that is, by a sequence of the following operations: (H0) adding a new vertex $z$ and edges $uz$, $vz$ for already existing, distinct vertices $u, v$, and (H1) deleting an already existing edge $uv$, and adding a new vertex $z$ and edges $uz$, $vz$, $wz$ for already existing, distinct vertices $u, v, w$. \cite{henneberg1911graphische, laman1970graphs}

\thmsparse*
\begin{proof}
Let $G=(V,E)$ be a simple graph such that its graphic matroid $M$ has a rainbow circuit-free coloring using each color at most twice. Let $X \subseteq V$ be a subset such that $|X|\ge 2$. As each rainbow colored set is independent in $M$, at most $r_M(E[X])$ colors appear in $E[X]$, thus $|E[X]| \le 2\cdot r_M(E[X]) \le 2\cdot (|X|-1) = 2\cdot |X|-2$. If $|E[X]| = 2\cdot |X|-2$ holds, then we get a rank-preserving rainbow circuit-free coloring of the restriction $M|E[X]$ such that each color class has size exactly two. As $M|E[X]$ is simple, this contradicts Corollary~\ref{cor:standard}. Hence $|E[X]| \le 2\cdot |X|-3$, and the graph is $(2,3)$-sparse as stated.
\end{proof}

By Theorem~\ref{thm:sparse}, each simple graph $G=(V,E)$ whose graphic matroid has a rainbow circuit-free coloring using each color at most twice has at most $2\cdot |V|-3$ edges. The next theorem characterizes simple graphs having such a coloring with exactly $2\cdot |V|-3$ edges. 

\thmho*
\begin{proof}
It is clear that if a graph can be constructed from $K_2$ by (H0) operations, then its graphic matroid has a rainbow circuit-free coloring using each color at most twice. Indeed, if we apply a (H0) operation to a graph and with a given rainbow circuit-free coloring and color the two new edges with the same new color, then we get a rainbow circuit-free coloring of the graphic matroid of the resulting graph.

To show the reverse direction, let $M$ be the graphic matroid of a simple graph $G$ which has such a coloring, and let $r$ denote the rank of $M$. As $|E|=2\cdot|V|-3 \ge 2\cdot r-1$ and each color is used at most twice, the coloring is rank-preserving. Therefore, by Theorem~\ref{thm:standard} and Lemma~\ref{lem:equiv}, there is an ordering $S_1, \dots, S_r$ of the color classes such that $S_i$ consists of parallel elements of $M/(S_1 \cups S_{i-1})$. As $M$ is simple, $|S_1| = 1$, hence $|S_2| = \dots =|S_r| = 2$. We claim that there is a labeling $v_0, v_1, \dots, v_r$ of $V$ such that $S_1 \cups S_i = E_G[\{v_0, \dots, v_i\}]$ for $i = 1, \dots, r$. Indeed, let $S_1 = \{v_0v_1\}$, and if we already defined $v_0, v_1, \dots, v_i$, then $S_{i+1}$ consists of two parallel edges of the graph obtained from the simple graph $G$ by the contraction of the vertices $\{v_0, \dots, v_i\}$. Hence the two edges of $S_i$ have a common vertex $v_{i+1} \not \in \{v_0,  \dots, v_i\}$ and the other endpoint of these edges are in $\{v_0, \dots, v_i\}$. Therefore, $G$ can be constructed by (H0) operations from the edge $v_0v_1$ such that the vertices are added in order $v_2, \dots, v_r$.  
\end{proof}

As a consequence, we get that the converse of Theorem~\ref{thm:sparse} does not hold as not every minimally rigid graph can be constructed using only (H0) operations.

\section*{Acknowledgements}

Krist\'of B\'erczi was supported by the Lend\"ulet Programme of the Hungarian Academy of Sciences -- grant number LP2021-1/2021 and by the Hungarian National Research, Development and Innovation Office -- NKFIH, grant number FK128673.
``Application Domain Specific Highly Reliable IT Solutions'' project  has been implemented with the support provided from the National Research, Development and Innovation Fund of Hungary, financed under the Thematic Excellence Programme TKP2020-NKA-06 (National Challenges Subprogramme) funding scheme.

\bibliographystyle{abbrv}
\bibliography{rainbow}

\end{document}